\documentclass[10pt,a4paper]{article}
\usepackage[latin1]{inputenc}
\usepackage{amsmath}
\usepackage{amsfonts}
\usepackage{amssymb}
\usepackage{graphicx}
\usepackage{amsthm} 	
\usepackage{tikz}
\usepackage{hyperref}
\usepackage{todonotes}
\usetikzlibrary{backgrounds}	
\usepackage[toc]{appendix}

\usepackage{graphicx,xcolor}	

\usepackage{xcolor}

\usepackage{esint}				

\usepackage[a4paper,top=2.1cm,bottom=2.60cm,left=3.8cm,right=3.8cm]{geometry} 

\allowdisplaybreaks

\newtheorem{theorem}{Theorem}[section]
\newtheorem{proposition}[theorem]{Proposition}
\newtheorem{corollary}[theorem]{Corollary}

\newtheorem{remark}[theorem]{Remark}
\newtheorem{lemma}[theorem]{Lemma}

\title{Multi-parameter analysis of the obstacle scattering problem}
\author{Matteo Dalla Riva\thanks{Dipartimento di Ingegneria, Universit\`a degli Studi di Palermo, Viale delle Scienze, Ed.~8, 90128 Palermo, Italy.} ,  Paolo Luzzini\thanks{Dipartimento di Matematica ``Tullio Levi Civita,'' Universit\`a degli Studi di Padova, Via Trieste 63, 35121 Padova, Italy.} , Paolo Musolino\thanks{Dipartimento di Scienze Molecolari e Nanosistemi, Universit\`a Ca' Foscari Venezia, Via Torino 155, 30172 Venezia Mestre, Italy}}

\date{October 11, 2021}

\begin{document}

\maketitle

\noindent
{\bf Abstract:}   %
 We consider the  acoustic field scattered by a bounded impenetrable obstacle and we study its dependence upon a certain set of parameters. As usual, the problem is modeled by  an exterior Dirichlet problem for the Helmholtz equation $\Delta u +k^2u=0$.  
We show that the solution $u$ and   its far field pattern $u_\infty$ depend real analytically on the shape of the obstacle,  
the wave number $k$, and the Dirichlet datum. We also prove a similar result  for the corresponding Dirichlet-to-Neumann map.
\vspace{9pt}

\noindent
{\bf Keywords:}  Helmholtz equation; acoustic scattering; associated  exterior Dirichlet problem; Dirichlet-to-Neumann operator; shape sensitivity analysis;  perturbed domain; integral equations.
\vspace{9pt}

\noindent   
{{\bf 2020 Mathematics Subject Classification:}} 35J25; 35J05; 35P25;  31B10; 45A05.

\section{Introduction}\label{sec:intro}

Understanding how the shape of an object impacts a certain property is a very old problem and has a variety of applications. We may think, for example, to the problem of finding the best design to maximize some sort of efficiency or to the sort of problems related to  {\em non-destructive test  methods},  like  the problem of finding the shape of an inclusion from a set of measurements taken on the outer boundary of an object, or the inverse scattering problem,  where the shape of an obstacle is inferred  from measurements of a scattered wave.    

In mathematics, the property that one wants to analyze is often associated with the solution of a  boundary value problem or  to a quantity related to the solution by a certain functional. Then, understanding how the shape impacts a specific property amounts to studying the dependence of the solution of the boundary value problem upon perturbations of the domain of the partial differential equation.

In mathematical jargon the problem of finding an optimal configuration that maximizes a shape functional goes under the name of  {\em shape optimization}. The reader may find some references  in the monographs by Henrot and Pierre \cite{HePi05}, Novotny and Soko\l owski \cite{NoSo13}, and Soko\l owski and Zol\'esio \cite{SoZo92}.  The problems of inferring a shape from measurements on the boundary of an outer domain or a scattered wave are known as  {\em inclusion detection} and {\em inverse scattering problems}, respectively, and are both  examples of  {\em inverse problems}. For some references we mention the books of Colton and Kress \cite{CoKr19} and Kirsch \cite{Ki21}.

  A preliminary  task that is common to shape optimization and the above mentioned inverse problems is that of understanding the regularity of the map that associates the shape of an object to the solution of the boundary value problem and to the specific quantity under consideration.  For most techniques, indeed,  it is desirable to have at least some sort of differentiability (as in  Kirsch  \cite{Ki93}, where the differentiability of the far field pattern is used in the numerical analysis of an inverse scattering problem). 
  
  It is not surprising, then,  that many papers deal with the differentiability properties of shape functionals and our paper is one of these. More specifically, we  examine an acoustic obstacle scattering problem and  study the dependence of the solution and of its far field pattern upon perturbations of the wave number, the Dirichlet datum, and the shape of the obstacle. We also consider  the pullback of the Dirichlet-to-Neumann operator and its dependence on the wave number and  the shape of the obstacle. 

Among the works  that precede our paper with similar results we mention those of Potthast \cite{Po94,Po96a,Po96b}, where the aim is  to prove that the layer potentials of the Helmholtz equation are Fr\'echet differentiable functions of the support of integration.  Potthast's results are obtained  in the framework of Schauder spaces and the final goal is that of analyzing the domain derivative of the far field pattern.  Related problems are studied in the papers of Haddar and Kress \cite{HaKr04}, Hettlich \cite{He95}, Kirsch \cite{Ki93}, and Kress and P\"aiv\"arinta \cite{KrPa99}. For similar  differentiability results, but  for the elastic scattering problem,  we mention Charalambopoulos \cite{Ch95}. Finally, the case of Lipschitz  domains have been studied by Costabel and Le Lou\"er \cite{CoLe12a, CoLe12b, Le12}  in the framework of Sobolev spaces.

The novelties that we bring in this list are of two kinds. On the one hand,  the regularity properties that we prove  are stronger than Fr\'echet differentiability. More specifically, we obtain {\em real analyticity} results. On the other hand, we do not confine ourselves only to  the shape of the obstacle, but we consider the joint regularity upon the wave number, the Dirichlet datum, and the shape. So, for example, we prove that the far field pattern is a real analytic map of the wave number,  the Dirichlet datum, and the shape of the obstacle  (a triple that we think as a unique variable in a certain product Banach space). 

Incidentally, we observe that there are very few results  in literature that go beyond the differentiability of shape functionals. A remarkable example are some  recent works on the {\it shape holomorphy} by  Jerez-Hanckes, Schwab, and Zech \cite{JeScZe17}, which deals with the electromagnetic wave scattering problem, by Cohen, Schwab, and Zech \cite{CoScZe18}, about the stationary Navier-Stokes equations, and by 
 Henr\'iquez and Schwab \cite{HeSc21}, on the Calder\'on projector for the Laplacian in $\mathbb{R}^2$.

We now introduce the geometry of the problem. We fix
\begin{equation}\label{Omega_def}
\begin{split}
&\text{$\alpha \in \mathopen]0,1[$ and a bounded open connected subset $\Omega$ of $\mathbb{R}^{3}$ of  class  $C^{1,\alpha}$}
\\
&\text{\hspace{2.5cm}such that $\mathbb{R}^{3}\setminus\overline{\Omega}$ is  connected.}
\end{split}
\end{equation}
Here we note that, if $\Omega$ is a set, the symbol $\overline{\Omega}$ denotes  its  closure. Also, if $z\in \mathbb{C}$, we denote by $\overline{z}$  the conjugate of the complex number $z$.  For the definition of sets and functions of the Schauder class $C^{j,\alpha}$ ($j \in \mathbb{N}$)  we refer, e.g., to Gilbarg and
Trudinger~\cite{GiTr83}. We also note that, if not otherwise specified, all the functions in the paper are complex-valued.

To consider perturbations of the shape of the obstacle, we take the set $\Omega$ of \eqref{Omega_def} as a reference set. Then we  introduce a specific class $\mathcal{A}^{1,\alpha}_{\partial \Omega}$ of $C^{1,\alpha}$-diffeomorphisms from $\partial\Omega$ to $\mathbb{R}^3$: $\mathcal{A}^{1,\alpha}_{\partial \Omega}$ is the set of functions of class $C^{1,\alpha}(\partial\Omega, \mathbb{R}^{3})$ that are injective and have injective differential at all points  of $\partial\Omega$. By Lanza de Cristoforis and Rossi \cite[Lem. 2.2, p. 197]{LaRo08}  
and \cite[Lem. 2.5, p. 143]{LaRo04},  we can see that $\mathcal{A}^{1,\alpha}_{\partial \Omega}$ is open 
in $ C^{1,\alpha}(\partial\Omega, \mathbb{R}^{3})$. Moreover, for all $\phi \in \mathcal{A}^{1,\alpha}_{\partial \Omega}$ the Jordan-Leray separation theorem ensures that 
$\mathbb{R}^{3}\setminus \phi(\partial \Omega)$  has exactly two open connected components (see, e.g., Deimling \cite[Thm.  5.2, p. 26]{De85}  and \cite[\S A.4]{DaLaMu21}). We denote  by $\mathbb{I}[\phi]$ the bounded connected component of $\mathbb{R}^{3}\setminus \phi(\partial \Omega)$ and by
$\mathbb{E}[\phi]$ the  unbounded one. Then, we have
\[
\mathbb{E}[\phi] =\mathbb{R}^{3} \setminus \overline{\mathbb{I}[\phi]}\quad\text{ and }\quad \overline{\mathbb{E}[\phi]} =\mathbb{R}^{3} \setminus \mathbb{I}[\phi]\, .
\]
The $\phi$-dependent set $\mathbb{I}[\phi]$ models the shape of the impenetrable obstacle (i.e. the scattering object), while $\mathbb{E}[\phi]=\mathbb{R}^{3} \setminus \overline{\mathbb{I}[\phi]}$ represents the homogeneous isotropic media where the scattered acoustic waves propagate (see Figure \ref{fig1}).


\begin{figure}[!htb]
\begin{center}
\includegraphics[width=5.1in]{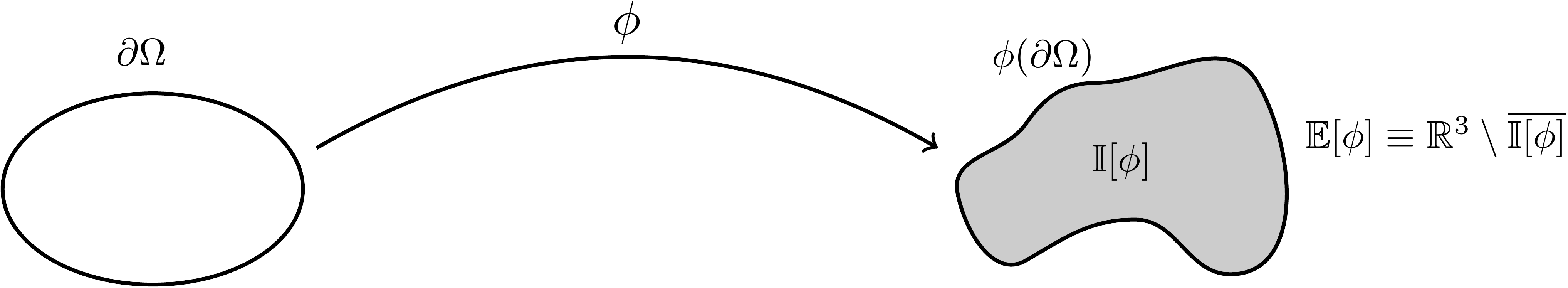}\\
\caption{{\it The reference set $\partial \Omega$, a diffeomorphism  $\phi \in \mathcal{A}^{1,\alpha}_{\partial \Omega}$, and the sets $\mathbb{I}[\phi]$ and $\mathbb{E}[\phi]$.}}\label{fig1}
\end{center}
\end{figure}

Now we take $\phi \in \mathcal{A}^{1,\alpha}_{\partial \Omega}$, $k\in \mathbb{C}$ with $\mathrm{Im} k\geq 0$, and $g \in C^{1,\alpha}(\partial \Omega)$. We consider the following direct obstacle scattering problem: we look for a (complex-valued) function   $u \in C^{1,\alpha}_{\mathrm{loc}}( \overline{\mathbb{E}[\phi]})$ such that
\begin{equation}\label{eq:obstaclepb}
\begin{cases}
\Delta u +k^2 u=0 & {\mathrm{in}}\ \mathbb{E}[\phi]\,,\\
u=g\circ \phi^{(-1)}&\mbox{on } \partial \mathbb{E}[\phi]\,, \\
\lim_{x\to \infty}|x| \Bigg (Du(x)\cdot \frac{x}{|x|}-iku(x)\Bigg)=0\,,\\
\lim_{x\to \infty}u(x)=0\,,
\end{cases}
\end{equation}
where, as usual, $$\Delta=\sum_{j=1}^3\frac{\partial^2 }{\partial x_j^2}\,.$$ The third condition in problem \eqref{eq:obstaclepb},   i.e.
\begin{equation}\label{eq:Som}
\lim_{x\to \infty}|x| \Bigg (Du(x)\cdot \frac{x}{|x|}-iku(x)\Bigg)=0\, ,
\end{equation}
is known as the {\it outgoing Sommerfeld $(k)$-radiation condition}. From the point of view of physics, solutions of the Helmholtz equation that satisfy the outgoing Sommerfeld condition describe waves that  scatter from a source situated in a bounded domain. In particular, waves with sources situated at infinity do not satisfy condition \eqref{eq:Som}. For $k\neq 0$, the Sommerfeld condition implies the decay at infinity of $u(x)$, and thus it is stronger than the last condition of problem \eqref{eq:obstaclepb} (cf., e.g.,  Colton and Kress \cite[Chap. 3, Rem. 3.4]{CoKr13}). For $k=0$, this is no more the case, as one can easily verify taking $u$ identically constant. In particular, for $k=0$ a solution  $u$ of \eqref{eq:obstaclepb} is a harmonic function  that, by the last condition of the system, is also harmonic at infinity (see Folland \cite[Chap. 2]{Fo95}).  Then, in this case it is the Sommerfeld condition to follow from the decay at infinity of $u(x)$ (see, e.g., Folland \cite[Prop. 2.75]{Fo95}). 

Either way, from the Sommerfeld condition if $k\in \mathbb{C}\setminus\{0\}$ and $\mathrm{Im} k\geq 0$, or from the decay of $u(x)$ if $k=0$, we can see that  problem \eqref{eq:obstaclepb} has a unique solution in $C^{1,\alpha}_{\mathrm{loc}}(\overline{\mathbb{E}[\phi]})$ for all choice of $\phi \in \mathcal{A}^{1,\alpha}_{\partial \Omega}$,  $k\in \mathbb{C}$ with $\mathrm{Im} k\geq 0$, and $g \in C^{1,\alpha}(\partial \Omega)$ (cf. Colton and Kress \cite[Chap. 3]{CoKr13} for the case with $k\neq 0$ and Folland \cite[Chap. 3]{Fo95} for $k=0$). From now on, we denote such a solution by $u[\phi,k,g]$.  

We stress that  we decided to state 
problem \eqref{eq:obstaclepb} including both the Sommerfeld condition and the decay at infinity for exactly this reason, that is, to have a unique solution both when $k\neq 0$ and $k=0$. Doing so we can study the dependence of the solution $u[\phi,k,g]$ upon the wave number $k\in \mathbb{C}$ with $\mathrm{Im} k\geq 0$ in a unified way, without the need of introducing two different problems for $k\neq 0$ and $k=0$. 

We also observe that there exists a function $u_\infty [\phi,k,g]$, defined on the boundary $\partial \mathbb{B}_3(0,1)$ of the three dimensional unit ball $\mathbb{B}_3(0,1)$ and with values in $\mathbb{C}$, for which we have the following asymptotic expansion:
\[
u[\phi,k,g](x)=\frac{e^{ik|x|}}{|x|}\left(u_\infty [\phi,k,g] \left(\frac{x}{|x|}\right)+O\left(\frac{1}{|x|}\right)\right)\quad\text{as }|x|\to +\infty\,.
\]
For $k\neq 0$,  $u_\infty [\phi,k,g]$ is known as the  {\it far field pattern} of $u[\phi,k,g]$ (see, e.g., Colton and Kress \cite[Chap. 3]{CoKr13})   and, for $k=0$, $u_\infty [\phi,k,g]$ is constant (it is indeed the  spherical harmonic of degree zero in the expansion $u[\phi,k,g](x)=\sum_{j=0}^\infty |x|^{-1-j} Y_j(x/|x|)$, where every $Y_j$ is a spherical harmonic of degree $j$).  Both for $k=0$ and $k\neq 0$, $u_\infty [\phi,k,g]$ can be  computed from the solution $u [\phi,k,g]$ by the formula 
\begin{align}\label{def:ff}
u_\infty [\phi,k,g] (x)= \frac{1}{4\pi}\int_{\partial \mathbb{B}_3(0,R)} \Bigg( u[\phi,k,g](y)\frac{\partial}{\partial \nu_{\mathbb{B}_3(0,R)}(y)}e^{-ikx\cdot y} -e^{-ikx\cdot y}\frac{\partial u[\phi,k,g](y)}{\partial \nu_{\mathbb{B}_3(0,R)}}  \Bigg)\, d&\sigma_y \\ \nonumber
\forall x \in \partial \mathbb{B}_3(0,1)& \, ,
\end{align}
where $R>0$ has to be taken large enough so that $\overline{\mathbb{I}[\phi]}\subseteq \mathbb{B}_3(0,R)$ and where $\nu_{ \mathbb{B}_3(0,R)}$ denotes the outward unit normal to $\partial \mathbb{B}_3(0,R)$. 
 By the divergence theorem, we can also verify that  the integral in the right-hand side of \eqref{def:ff} does not depend on the specific choice of $R$. From the point of view of physics, the far field pattern represents the main directional (angular) part of a wave away from a  scattering object. In inverse scattering theory, one of the main problems is that of reconstructing the properties of an object starting from the knowledge of the far field pattern. 

Moreover, if $\phi \in \mathcal{A}^{1,\alpha}_{\partial \Omega}$,  $k\in \mathbb{C}$ with $\mathrm{Im} k\geq 0$, we introduce the pullback of the Dirichlet-to-Neumann operator $\mathcal{D}_{(\phi,k)}$ from $C^{1,\alpha}(\partial \Omega)$ to $C^{0,\alpha}(\partial \Omega)$ as the linear operator that takes the Dirichlet datum $g$ to the normal derivative of the solution $u[\phi,k,g]$, i.e.  
\[
\mathcal{D}_{(\phi,k)}[g]\equiv \Big( \frac{\partial}{\partial \nu_{\mathbb{I}[\phi]}}u[\phi,k,g]\Big)\circ \phi\, .
\]

Our aim is to investigate the dependence of the solution $u[\phi,k,g]$  and of its far field pattern  $u_\infty [\phi,k,g]$  upon the triple $(\phi, k, g)$, and of the Dirichlet-to-Neumann operator $\mathcal{D}_{(\phi,k)}$ upon the pair $(\phi,k)$. 
As mentioned above, the rationale of this paper is to prove regularity properties that go beyond the Fr\'echet differentiability. More specifically, we do not confine to the dependence  on the shape: we study the joint dependence on the triple $(\phi,k,g)$ and we prove  (joint) {\em real analyticity} results. So, for example, in Theorem \ref{thm:an3}  we show that the map 
\[
\left(\mathcal{A}^{1,\alpha}_{\partial \Omega} \times \mathbb{C}^+ \times  C^{1,\alpha}(\partial \Omega)\right)\ni (\phi,k,g)\mapsto u_\infty[\phi,k,g] \in C^2(\partial \mathbb{B}_3(0,1))
\]
is real analytic. In the expression above $\mathbb{C}^+$ is the set of complex numbers $k$ with $\mathrm{Im}k\ge 0$. In Theorems \ref{thm:an1} and \ref{thm:an2}, we prove similar results also for the solution $u[\phi,k,g]$ and for its normal derivative. In Corollary \ref{cor:an} we deduce by Theorem \ref{thm:an2} a corresponding result for the pullback of the Dirichlet-to-Neumann map.

We stress here that for us the word ``analytic"  always means   ``real analytic." For the definition and properties of real analytic operators, we refer to Deimling \cite[\S 15]{De85}.

Our analysis relies on the results of  \cite{DaLa10a}, where the authors consider layer potentials associated with a family of fundamental solutions of second order differential operators with constant coefficients depending on a parameter. The authors prove the real analytic dependence of the layer potentials upon variations of the diffeomorphism, the density, and the parameter. In the present paper we apply the results of \cite{DaLa10a} to the $k$-dependent fundamental solution 
$
- \frac{1}{4\pi |x|}e^{ik|x|} 
$, $x \in \mathbb{R}^3
\setminus \{0\}$, of the Helmholtz equation $\Delta u +k^2 u=0$. 
We also mention the  work of  Lanza de Cristoforis and Rossi \cite{LaRo08} on layer potentials of the Helmholtz equation, where they consider a different family of fundamental solutions (see also the previous work \cite{LaRo04} by the same authors, which deals with harmonic layer potentials and \cite{Da08} for the case of higher order operators). Moreover, analyticity results for integral operators and methods of potential theory have been exploited in the monograph \cite{DaLaMu21} to obtain real analytic continuation properties of the solutions of singularly perturbed boundary value problems. Finally, we point out that an analysis similar to the one of the present paper has been carried out by the authors for other physical quantities arising in fluid mechanics and in material science (see \cite{DaLuMuPu21, LuMu20, LuMuPu19} for the longitudinal fluid flow along a periodic array of cylinders and the effective conductivity of a periodic two-phase composite material).

 The paper is organized as follows. Section \ref{s:prel} is a section of preliminaries of classical potential theory for the Helmholtz equation.
 In Section \ref{s:inteqfor} we transform problem  \eqref{eq:obstaclepb} into an equivalent integral equation. Finally, in Section \ref{s:functionals} we prove our main results on the analyticity of functions related to problem \eqref{eq:obstaclepb}. 
 

\section{Preliminaries of potential theory}\label{s:prel}

Let $\alpha \in \mathopen]0,1[$ and let $\tilde{\Omega}$ be a bounded open connected subset of $\mathbb{R}^{3}$ of  class  $C^{1,\alpha}$. We denote by $\nu_{\tilde{\Omega}}$ the outward unit normal to $\partial\tilde{\Omega}$ and by $d\sigma$ the area element on $\partial\tilde{\Omega}$.

We remark that, in this section,  $\tilde{\Omega}$ is a generic open subset of $\mathbb{R}^3$ that we use as a dummy to define some notation and write some general results.  Instead, the set $\Omega$ introduced in \eqref{Omega_def} is a reference domain  that we keep fixed for the whole paper.


Our method is based on classical potential theory. In order to construct layer potentials,  we introduce 
 for $k \in \mathbb{C}$ the function 
\[
S(k,x)\equiv
-\frac{1}{4\pi |x|}e^{ik|x|} \qquad \forall x \in \mathbb{R}^3 \setminus\{0\} \,.
\]
For $k\neq 0$, $S(k,x)$ is  a standard fundamental solution of the Helmholtz equation $\Delta u+k^2u=0$ that satisfies the outgoing Sommerfeld $(k)$-radiation condition.  For $k=0$, $S(0,x)$ is  a standard fundamental solution of the Laplace equation $\Delta u=0$,  that is
\[
S(0,x)=-\frac{1}{4\pi |x|} \qquad \forall x \in \mathbb{R}^3 \setminus\{0\}\, ,
\]
and  is harmonic at infinity.

Then, we introduce the layer potentials associated  with the fundamental solution $S(k,\cdot)$. We set
\begin{eqnarray*}
v[\partial\tilde{\Omega},k, \mu](x)&\equiv&
\int_{\partial\tilde{\Omega}}S(k,x-y)\mu(y)\,d\sigma_{y}\qquad\forall x\in \mathbb{R}^{3}\,,
\\
w[\partial\tilde{\Omega},k, \mu](x)&\equiv&
\int_{\partial\tilde{\Omega}}
 \frac{\partial}{\partial \nu_{\tilde{\Omega}}(y)}S(k,x-y)\mu(y)\,d\sigma_{y}
\\
&\equiv&
-\int_{\partial\tilde{\Omega}}
 \nu_{\tilde{\Omega}}(y) \cdot 
DS(k,x-y)\mu(y)\,d\sigma_{y}\qquad\forall x\in \mathbb{R}^{3}\,,
\end{eqnarray*}
for all $\mu\in C^0(\partial\tilde{\Omega})$.   Here above, $DS(k,\xi)$ denotes the gradient of $S(k,\cdot)$ computed at the point $\xi \in \mathbb{R}^{3} \setminus \{0\}$. 
We also clarify that, in this paper, 
\[
\frac{\partial}{\partial \nu_{\tilde{\Omega}}(y)}
\]
is the partial derivative (in the normal direction) with respect to the $y$ variable, whereas 
\[
\frac{\partial}{\partial \nu_{\tilde{\Omega}}(x)}
\] 
denotes the partial derivative with respect to the $x$ variable. This is why  a $-$ (minus) sign appears in front of the last integral.  
We also set
\begin{eqnarray*}
V[\partial\tilde{\Omega},k, \mu](x)&\equiv&
\int_{\partial\tilde{\Omega}}S(k,x-y)\mu(y)\,d\sigma_{y}\qquad\forall x\in \partial \tilde{\Omega}\,,
\\
W[\partial\tilde{\Omega},k, \mu](x)&\equiv&
\int_{\partial\tilde{\Omega}}
 \frac{\partial}{\partial \nu_{\tilde{\Omega}}(y)}S(k,x-y)\mu(y)\,d\sigma_{y}\qquad\forall x\in \partial \tilde{\Omega}\,,
 \\
W^\ast[\partial\tilde{\Omega},k, \mu](x)&\equiv&
\int_{\partial\tilde{\Omega}}
 \frac{\partial}{\partial \nu_{\tilde{\Omega}}(x)}S(k,x-y)\mu(y)\,d\sigma_{y}\qquad\forall x\in \partial \tilde{\Omega}\,,
\end{eqnarray*}
for all $\mu\in C^0(\partial\tilde{\Omega})$.  
The function $v[\partial\tilde{\Omega},k,\mu]$ is called ``single  layer potential" and $w[\partial\tilde{\Omega},k,\mu]$ is called ``double  layer potential."  
As is well known, if $\mu\in C^{0}(\partial\tilde{\Omega})$, then $v[\partial\tilde{\Omega},k,\mu]$ is  continuous in  
$\mathbb{R}^{3}$ and we set
\[
v^{+}[\partial\tilde{\Omega},k, \mu] \equiv  v[\partial\tilde{\Omega},k,\mu]_{
|\overline{\tilde{\Omega}}}
\quad\text{and}\quad
  v^{-}[\partial\tilde{\Omega},k, \mu]\equiv  v[\partial\tilde{\Omega},k,\mu]_{
|\mathbb{R}^{3} \setminus \tilde{\Omega}}\, .
\]
In  Theorems  \ref{slpthm} and \ref{dlpthm} below, we collect some  well-known properties of layer potentials (cf., e.g.,  \cite{DaLa10a} with Lanza de Cristoforis, Colton and Kress \cite{CoKr13}, Lanza de Cristoforis and Rossi \cite{LaRo04, LaRo08}).

\begin{theorem}\label{slpthm}
Let $\alpha \in \mathopen]0,1[$ and let $\tilde{\Omega}$ be a bounded open connected subset of $\mathbb{R}^{3}$ of  class  $C^{1,\alpha}$. 
  Let $k\in \mathbb{C}$ be such that $\mathrm{Im}k\geq 0$. Then the following statements hold.
\begin{itemize}
\item[(i)] If $\mu \in C^{0,\alpha}(\partial \tilde{\Omega})$, then $v^+[\partial \tilde{\Omega}, k, \mu] \in C^{1,\alpha}(\overline{\tilde{\Omega}})$ and $v^-[\partial \tilde{\Omega}, k, \mu] \in C^{1,\alpha}_{\mathrm{loc}}(\mathbb{R}^{3} \setminus \tilde{\Omega})$. Moreover,
\[
\Delta v[\partial \tilde{\Omega}, k, \mu] +k^2 v[\partial \tilde{\Omega}, k, \mu]=0 \qquad \mbox{in $\mathbb{R}^{3} \setminus \partial \tilde{\Omega}$} 
\]
and
\begin{itemize}
\item[$\bullet$] if $k\neq 0$, then $v^-[\partial \tilde{\Omega}, k, \mu]$ satisfies the outgoing Sommerfeld $(k)$-radiation condition \eqref{eq:Som},
\item[$\bullet$] if $k=0$, then $v^-[\partial \tilde{\Omega}, k, \mu]$ is harmonic at infinity.
\end{itemize}
\item[(ii)] The map from $C^{0,\alpha}(\partial\tilde{\Omega})$ to $C^{1,\alpha}(\overline{\tilde{\Omega}})$ 
that takes $\mu$ to  $v^{+}[\partial\tilde{\Omega}, k,  \mu]$ is linear and continuous. 
 If $R>0$ is such that $\overline{\tilde{\Omega}}\subseteq \mathbb{B}_3(0,R)$, then the map from $C^{0,\alpha}(\partial\tilde{\Omega})$ to $C^{1,\alpha}(\overline{\mathbb{B}_3(0,R)} \setminus \tilde{\Omega})$ 
that takes $\mu$ to  $v^{-}[\partial\tilde{\Omega},k, \mu]_{|\overline{\mathbb{B}_3(0,R)} \setminus \tilde{\Omega}}$ is linear and continuous. 
\item[(iii)] Let  $\mu \in C^{0,\alpha}(\partial\tilde{\Omega})$. Then 
\[
\frac{\partial}{\partial \nu_{\tilde{\Omega}}} v^\pm[\partial\tilde{\Omega},k, \mu] = \mp \frac{1}{2}\mu +  
W^{\ast}[\partial \tilde{\Omega},k, \mu] \qquad \mbox{ on } \partial \tilde{\Omega}\, .
\]
Moreover, the map that takes $\mu \in C^{0,\alpha}(\partial \tilde{\Omega})$ to $W^{\ast}[\partial \tilde{\Omega},k, \mu]$ is a compact operator from $C^{0,\alpha}(\partial \tilde{\Omega})$ to itself.
\end{itemize}
\end{theorem}

\begin{theorem}\label{dlpthm}
Let $\alpha \in \mathopen]0,1[$ and let $\tilde{\Omega}$ be a bounded open connected subset of $\mathbb{R}^{3}$ of  class  $C^{1,\alpha}$. 
 Let  $k\in \mathbb{C}$  be such that $\mathrm{Im}k\geq 0$. Then the following statements hold.
\begin{itemize}
\item[(i)] If $\mu \in C^{1,\alpha}(\partial \tilde{\Omega})$, then $w[\partial \tilde{\Omega}, k, \mu]_{|\tilde{\Omega}}$ can be extended to a continuous function $w^+[\partial \tilde{\Omega}, k, \mu] \in C^{1,\alpha}(\overline{\tilde{\Omega}})$ and $w[\partial \tilde{\Omega}, k, \mu]_{|\mathbb{R}^3 \setminus \overline{\tilde{\Omega}}}$ can be extended to a continuous function $w^-[\partial \tilde{\Omega}, k, \mu] \in C^{1,\alpha}_{\mathrm{loc}}(\mathbb{R}^{3} \setminus \tilde{\Omega})$. Moreover,
\[
\Delta w[\partial \tilde{\Omega}, k, \mu] +k^2 w[\partial \tilde{\Omega}, k, \mu]=0 \qquad \mbox{in $\mathbb{R}^{3} \setminus \partial \tilde{\Omega}$}
\]
and
\begin{itemize}
\item[$\bullet$] if $k\neq 0$, then $w^-[\partial \tilde{\Omega}, k, \mu]$ satisfies the outgoing Sommerfeld $(k)$-radiation condition \eqref{eq:Som},
\item[$\bullet$] if $k=0$, then $w^-[\partial \tilde{\Omega}, k, \mu]$ is harmonic at infinity.
\end{itemize}
\item[(ii)] The map from $C^{1,\alpha}(\partial\tilde{\Omega})$ to $C^{1,\alpha}(\overline{\tilde{\Omega}})$ 
that takes $\mu$ to  $w^{+}[\partial\tilde{\Omega}, k,  \mu]$ is linear and continuous. 
 If $R>0$ is such that $\overline{\tilde{\Omega}}\subseteq \mathbb{B}_3(0,R)$, then the map from $C^{1,\alpha}(\partial\tilde{\Omega})$ to $C^{1,\alpha}(\overline{\mathbb{B}_3(0,R)} \setminus \tilde{\Omega})$ 
that takes $\mu$ to  $w^{-}[\partial\tilde{\Omega},k, \mu]_{|\overline{\mathbb{B}_3(0,R)} \setminus \tilde{\Omega}}$ is linear and continuous. 
\item[(iii)] Let  $\mu \in C^{1,\alpha}(\partial\tilde{\Omega})$. Then 
\[
w^\pm[\partial\tilde{\Omega},k, \mu] = \pm \frac{1}{2}\mu +  
W[\partial \tilde{\Omega},k, \mu] \qquad \mbox{ on } \partial \tilde{\Omega}\, 
\]
and
\[
\frac{\partial}{\partial \nu_{\tilde{\Omega}}} w^+[\partial\tilde{\Omega},k, \mu] = \frac{\partial}{\partial \nu_{\tilde{\Omega}}} w^-[\partial\tilde{\Omega},k, \mu] \qquad \mbox{ on } \partial \tilde{\Omega}\, .
\]
Moreover, the map that takes $\mu \in C^{1,\alpha}(\partial \tilde{\Omega})$ to $W[\partial \tilde{\Omega},k, \mu]$ is a compact operator from $C^{1,\alpha}(\partial \tilde{\Omega})$ to itself.
\end{itemize}
\end{theorem}

 Our approach is based on integral equations. More precisely, in order to study problem \eqref{eq:obstaclepb}, we convert it into an equivalent integral equation. We do so by exploiting a representation formula of the solution $u[\phi,k,g]$ in terms of  single and  double layer potentials. Therefore, we now show the validity of the following variant of the result of Colton and Kress \cite[Thm.~3.33]{CoKr13} 
regarding 
the solvability of the exterior Dirichlet problem for the Helmholtz equation by means of a combined double and single layer potential.

\begin{theorem}\label{thm:varexist}
Let $\alpha \in \mathopen]0,1[$ and let $\tilde{\Omega}$ be a bounded open connected subset of $\mathbb{R}^{3}$ of  class  $C^{1,\alpha}$ such that $\mathbb{R}^{3}\setminus\overline{\tilde{\Omega}}$ is  connected. Let $k\in \mathbb{C}$ be such that $\mathrm{Im}k\geq 0$. Then the following statements hold.
\begin{itemize}
\item[(i)] The integral operator $T$ from $C^{1,\alpha}(\partial \tilde{\Omega})$ to itself defined by
\[
T\equiv -\frac{1}{2}I+W[\partial \tilde{\Omega}, k,\cdot]+(1-i \mathrm{Re} k)  V[\partial \tilde{\Omega},k,\cdot],
\]
where $I$ denotes the identity operator, is a linear homeomorphism.
\item[(ii)] Let $\Gamma \in C^{1,\alpha}(\partial \tilde{\Omega})$. Then problem
\begin{equation}\label{eq:obstaclepbbis}
\begin{cases}
\Delta u +k^2 u=0 &  \mbox{in}\ \mathbb{R}^{3} \setminus \overline{\tilde{\Omega}}\,,\\
u=\Gamma &\mbox{on } \partial \tilde{\Omega}\,, \\
\lim_{x\to \infty}|x| \Bigg (Du(x)\cdot \frac{x}{|x|}-iku(x)\Bigg)=0\,,\\
\lim_{x\to \infty}u(x)=0
\end{cases}
\end{equation}
has a unique solution $u \in C^{1,\alpha}_{\mathrm{loc}}(\mathbb{R}^{3} \setminus \tilde{\Omega})$. Moreover,
\[
u=w^-[\partial \tilde{\Omega}, k,\mu]+(1-i \mathrm{Re} k) v^-[\partial \tilde{\Omega},k,\mu] \quad  \mbox{ in } \mathbb{R}^3 \setminus \tilde \Omega \, ,
\]
where $\mu \in C^{1,\alpha}(\partial \Omega)$ is delivered by
\[
\mu =\Big(-\frac{1}{2}I+W[\partial \tilde{\Omega}, k,\cdot]+(1-i \mathrm{Re} k) V[\partial \tilde{\Omega},k,\cdot]\Big)^{(-1)}[\Gamma]\, .
\]
\end{itemize}
\end{theorem}
\begin{proof}
We first consider statement (i). We modify the proof of Colton and Kress \cite[Thm.~3.33]{CoKr13}. We first note that
by Theorems \ref{slpthm} (ii) and \ref{dlpthm} (iii), by the continuity of the single layer potential, by the compactness of the embedding of 
$C^{1,\alpha}(\partial\tilde\Omega)$ in $C^{0,\alpha}(\partial\tilde\Omega)$, and by the continuity of the restriction 
operator from $C^{1,\alpha}(\overline{\tilde \Omega})$ to $C^{1,\alpha}(\partial\tilde\Omega)$, the operator 
\[
\psi \mapsto W[\partial \tilde{\Omega}, k,\psi]+(1-i \mathrm{Re} k) V[\partial \tilde{\Omega},k,\psi]
\]
is compact from $C^{1,\alpha}(\partial \tilde{\Omega})$ to itself. Therefore,
\[
T   = -\frac{1}{2}I+W[\partial \tilde{\Omega}, k,\cdot]+(1-i \mathrm{Re} k) V[\partial \tilde{\Omega},k,\cdot]
\]
is a Fredholm operator of index $0$. As a consequence, to show that $T$ invertible, it suffices to prove that it is injective. So let $\psi \in C^{1,\alpha}(\partial \tilde{\Omega})$ be such that
\[
 -\frac{1}{2}\psi+W[\partial \tilde{\Omega}, k,\psi]+(1-i \mathrm{Re} k) V[\partial \tilde{\Omega},k,\psi]=0\, .
\]
Then, by the continuity of the single layer potential and  by the jump formula for the double layer potential (see  
Theorem \ref{dlpthm} (iii)), the function $u \in C_{\mathrm{loc}}^{1,\alpha}(\mathbb{R}^3 \setminus \tilde \Omega)$ defined by
\[
u=w^-[\partial \tilde{\Omega}, k,\psi]+(1-i \mathrm{Re} k) v^-[\partial \tilde{\Omega},k,\psi] \quad  \mbox{ in } \mathbb{R}^3 \setminus \tilde \Omega
\]
solves the homogeneous exterior Dirichlet problem 
\[
\begin{cases}
\Delta u +k^2 u=0 & {\mathrm{in}}\ \mathbb{R}^{3} \setminus \overline{\tilde{\Omega}}\,,\\
u=0 &\mbox{on } \partial \tilde{\Omega}\,, \\
\lim_{x\to \infty}|x| \Bigg (Du(x)\cdot \frac{x}{|x|}-iku(x)\Bigg)=0\, ,\\
\lim_{x\to \infty}u(x)=0\, ,
\end{cases}
\]
and thus,  by the uniqueness of the solution of problem \eqref{eq:obstaclepbbis}  (cf. Colton and Kress \cite[Chap. 3]{CoKr13} for the case with $k\neq 0$ and Folland \cite[Chap. 3]{Fo95} for $k=0$), we have
\[
u=0 \qquad \mbox{in } \  \mathbb{R}^3 \setminus {\tilde{\Omega}}\,.
\]
 Next we set
\[
u^\#\equiv w^+[\partial \tilde{\Omega}, k,\psi]+(1-i \mathrm{Re} k) v^+[\partial \tilde{\Omega},k,\psi] \quad  \mbox{ in } \overline{\tilde{\Omega}}  \,.
\]
Clearly,  $u^\# \in C^{1,\alpha}(\overline{\tilde{\Omega}})$ and  by the jump relations for the layer potentials (see Theorems \ref{slpthm} and \ref{dlpthm}), we have
\[
u^\#_{|\partial\tilde\Omega}=\psi+w^-[\partial \tilde{\Omega}, k,\psi]_{|\partial\tilde\Omega}+(1-i \mathrm{Re} k) v^-[\partial \tilde{\Omega},k,\psi]_{|\partial\tilde\Omega}=\psi
\]
and
\begin{align*}
\frac{\partial}{\partial \nu_{\tilde{\Omega}}}u^\#
&=\frac{\partial}{\partial \nu_{\tilde{\Omega}}}w^-[\partial \tilde{\Omega}, k,\psi]+(1-i \mathrm{Re} k) \left(-\psi+ \frac{\partial}{\partial \nu_{\tilde{\Omega}}}v^-[\partial \tilde{\Omega},k,\psi]\right)\\
&=(-1+i \mathrm{Re} k) \psi\, . 
\end{align*}
Then, the first Green identity (cf., e.g., Colton and Kress \cite[(3.4), p.~68]{CoKr13}) implies that
\begin{equation}\label{eq:varexist:1}
(-1+i \mathrm{Re} k) \int_{\partial \tilde{\Omega}}|\psi|^2d\sigma=\int_{\partial \tilde{\Omega}}\overline{u^\#}\frac{\partial u^\#}{\partial \nu_{\tilde{\Omega}}}d\sigma=\int_{\tilde{\Omega}}|\nabla u^\#|^2-k^2|u^\#|^2\, dx\, .
\end{equation}
Taking the real part in \eqref{eq:varexist:1}, we obtain
\begin{equation}\label{eq:varexist:2}
- \int_{\partial \tilde{\Omega}}|\psi|^2d\sigma=\int_{\tilde{\Omega}}|\nabla u^\#|^2-[(\mathrm{Re}k)^2-(\mathrm{Im}k)^2]|u^\#|^2\, dx\,
\end{equation}
and taking the imaginary part in \eqref{eq:varexist:1}, we get
\begin{equation}\label{eq:varexist:3}
\mathrm{Re}k \int_{\partial \tilde{\Omega}}|\psi|^2d\sigma=-2(\mathrm{Re}k)(\mathrm{Im}k)\int_{\tilde{\Omega}}|u^\#|^2\, dx\, .
\end{equation}

Now, if $\mathrm{Re}k\neq0$, then equation \eqref{eq:varexist:3} implies that $\psi=0$ (we also remember that $\mathrm{Im}k\ge 0$) and if  $\mathrm{Re}k=0$, then equation \eqref{eq:varexist:2} implies that $\psi=0$ .  Either way, we have $\psi=0$ and statement (i) follows. The validity of statement (ii) follows from statement (i), from the jump formulas for the double layer potential of Theorems \ref{dlpthm}, and from the continuity of the single layer potential.
\end{proof}

We now introduce a technical lemma about the real analytic dependence upon the diffeomorphism $\phi$ of some maps related to the change of variables in integrals and to the outer normal field. For a proof we refer to Lanza de Cristoforis and Rossi \cite[p.~166]{LaRo04}
and to Lanza de Cristoforis \cite[Prop. 1]{La07}.
\begin{lemma}\label{lem:rajacon}
 Let $\alpha$, $\Omega$ be as in \eqref{Omega_def}.  Then the following statements hold.
\begin{itemize}
\item[(i)] For each $\phi \in \mathcal{A}^{1,\alpha}_{\partial \Omega}$, there exists a unique  
$\tilde \sigma[\phi] \in C^{0,\alpha}(\partial\Omega)$ such that $\tilde \sigma[\phi] > 0$ and 
\[ 
  \int_{\phi(\partial\Omega)}w(s)\,d\sigma_s=  \int_{\partial\Omega}w \circ \phi(y)\tilde\sigma[\phi](y)\,d\sigma_y \qquad \forall w \in L^1(\phi(\partial\Omega)).
\]
Moreover, the map $\tilde \sigma[\cdot]$ from $\mathcal{A}^{1,\alpha}_{\partial \Omega}  $ to $ C^{0,\alpha}(\partial\Omega)$ is real analytic.
\item[(ii)] The map from $\mathcal{A}^{1,\alpha}_{\partial \Omega} $ to $ C^{0,\alpha}(\partial\Omega, \mathbb{R}^{3})$ that takes $\phi$ to $\nu_{\mathbb{I}[\phi]} \circ \phi$ is real analytic.
\end{itemize}
\end{lemma}

 By the results of \cite{DaLa10a} and the definition of $S(k,\cdot)$, we deduce the following lemma on the real analyticity of some maps related to the $\phi$-pullback of layer potentials and their derivatives (see also Lanza de Cristoforis and Rossi \cite{LaRo08} and Lanza de Cristoforis \cite[\S 3]{La12}).

\begin{lemma}\label{lem:anSLP}
Let $\alpha$, $\Omega$ be as in \eqref{Omega_def}.  
Then the following statements hold.  
\begin{itemize}
 \item[(i)] The map from $\mathcal{A}^{1,\alpha}_{\partial \Omega}\times\mathbb{C} \times C^{0,\alpha}(\partial\Omega)$ to  $C^{1,\alpha}(\partial\Omega)$ that takes a triple $(\phi,k,\theta)$ to the function $V[\partial \mathbb{I}[\phi],k,\theta\circ \phi^{(-1)}]\circ \phi$ is real analytic.
 \item[(ii)] The map from $\mathcal{A}^{1,\alpha}_{\partial \Omega}\times  \mathbb{C} \times C^{1,\alpha}(\partial\Omega)$ to  $C^{1,\alpha}(\partial\Omega)$ that takes a triple $(\phi,k,\theta)$ to the function $W[\partial \mathbb{I}[\phi],k,\theta\circ \phi^{(-1)}]\circ \phi$ is real analytic.
\item[(iii)] The map from $\mathcal{A}^{1,\alpha}_{\partial \Omega}\times  \mathbb{C} \times C^{0,\alpha}(\partial\Omega)$ to  $C^{0,\alpha}(\partial\Omega)$ that takes a triple $(\phi,k,\theta)$ to the function $W^{\ast}[\partial \mathbb{I}[\phi],k,\theta\circ \phi^{(-1)}]\circ \phi$ is real analytic.
\item[(iv)] The map from $\mathcal{A}^{1,\alpha}_{\partial \Omega}\times  \mathbb{C} \times C^{1,\alpha}(\partial\Omega)$ to  $C^{0,\alpha}(\partial\Omega)$ that takes a triple $(\phi,k,\theta)$ to the function 
\[
\Big( \frac{\partial}{\partial \nu_{\mathbb{I}[\phi]}}w^-[\partial \mathbb{I}[\phi],k,\theta\circ \phi^{(-1)}]\Big)\circ \phi
\]
 is real analytic.
\end{itemize}
\end{lemma}
   \begin{proof}
   By a straightforward computation, one verifies that the $k$-dependent families of fundamental solutions $S(k,\cdot)$ and of differential operators $P[k](u)\equiv\Delta u+ k^2 u$ satisfy the assumption in \cite[(1.1)]{DaLa10a}. Then the validity of statements (i)--(iv) follows by \cite[Thm.~5.6]{DaLa10a}.
   \end{proof}


 \section{Analysis of the integral equation formulation of problem (\ref{eq:obstaclepb})} \label{s:inteqfor}
 
 By Theorem \ref{thm:varexist} we can transform problem  \eqref{eq:obstaclepb} into an equivalent integral equation. Then, the dependence of the solution of problem \eqref{eq:obstaclepb} on the shape of the obstacle, the wave number, and the Dirichlet datum, can be analyzed studying the dependence of the solution of the equivalent integral equation on the triple $(\phi,k,g)$. We begin with the following Proposition \ref{prop:existchange}, which follows from Theorem \ref{thm:varexist} and from a change of variable.

 \begin{proposition}\label{prop:existchange}
 Let $\alpha$, $\Omega$ be as in \eqref{Omega_def}. Let $\phi \in\mathcal{A}^{1,\alpha}_{\partial \Omega}$. Let  $k\in \mathbb{C}$  be such that $\mathrm{Im} k\geq 0$. Let $g \in C^{1,\alpha}(\partial \Omega)$. Then the unique solution $u[\phi,k,g] \in C^{1,\alpha}_{\mathrm{loc}}(\overline{\mathbb{E}[\phi]})$ of problem \eqref{eq:obstaclepb} is delivered by
\[
u[\phi,k,g]=w^-[\partial \mathbb{I}[\phi], k,\theta \circ \phi^{(-1)}]+(1-i \mathrm{Re} k) v^-[\partial \mathbb{I}[\phi],k,\theta\circ \phi^{(-1)}]\, ,
\]
where $\theta \in C^{1,\alpha}(\partial \Omega)$ is the unique solution of
\begin{equation}\label{existchange1}
-\frac{1}{2}\theta+W[\partial \mathbb{I}[\phi], k,\theta \circ \phi^{(-1)}]\circ \phi+(1-i \mathrm{Re} k) V[\partial \mathbb{I}[\phi],k,\theta \circ \phi^{(-1)}]\circ \phi =g\, .
\end{equation}
 \end{proposition}
In view of the previous Proposition \ref{prop:existchange},  we  find convenient to introduce for all $(\phi,k) \in\mathcal{A}^{1,\alpha}_{\partial \Omega} \times \mathbb{C}$ the auxiliary operator
 \[
 \Lambda(\phi,k):C^{1,\alpha}(\partial \Omega) \to C^{1,\alpha}(\partial \Omega)
 \]
 defined by setting
 \[
\Lambda(\phi,k)[\theta]\equiv -\frac{1}{2}\theta+W[\partial \mathbb{I}[\phi], k,\theta \circ \phi^{(-1)}]\circ \phi+(1-i \mathrm{Re} k) V[\partial \mathbb{I}[\phi],k,\theta \circ \phi^{(-1)}]\circ \phi
 \]
 for all $\theta\in C^{1,\alpha}(\partial \Omega)$.  Then, we can rewrite the integral equation \eqref{existchange1} as
 \begin{equation}\label{Linteq}
 \Lambda(\phi,k)[\theta]=g.
 \end{equation}
  We plan to show that \eqref{Linteq}  has a unique solution $\theta[\phi,k,g]$ that  depends analytically on $(\phi,k,g)$. To do so, we will show that the map that takes $(\phi,k)$ to $ \Lambda(\phi,k)$ is real analytic and invertible and then we will exploit the real analyticity of the inversion map and the formula
 \[
 \theta[\phi,k,g]=\Lambda(\phi,k)^{(-1)}[g]\,.
 \]
We begin by proving that $(\phi,k)\mapsto \Lambda(\phi,k)$ is real analytic from $\mathcal{A}^{1,\alpha}_{\partial \Omega} \times \mathbb{C}$ to the space 
\[
\mathcal{L}\left(C^{1,\alpha}(\partial \Omega),C^{1,\alpha}(\partial \Omega)\right)
\]
of linear bounded operators from $C^{1,\alpha}(\partial \Omega)$ to itself equipped, as usual, with the operator norm.

  \begin{proposition}\label{prop:Lmbdan}
   Let $\alpha$, $\Omega$ be as in \eqref{Omega_def}. Then the map that takes $(\phi,k)\in \mathcal{A}^{1,\alpha}_{\partial \Omega} \times \mathbb{C}$ to  $\Lambda(\phi,k)\in\mathcal{L}\left(C^{1,\alpha}(\partial \Omega),C^{1,\alpha}(\partial \Omega)\right)$ is real analytic.
     \end{proposition}
     \begin{proof}
      By Lemma \ref{lem:anSLP} the maps
    \[
    (\phi,k,\theta) \mapsto V[\partial \mathbb{I}[\phi],k,\theta\circ \phi^{(-1)}]\circ \phi 
    \] 
      from $\mathcal{A}^{1,\alpha}_{\partial \Omega}\times \mathbb{C}\times C^{0,\alpha}(\partial\Omega)$ to  $C^{1,\alpha}(\partial\Omega)$ and 
            \[
      (\phi,k,\theta) \mapsto W[\partial \mathbb{I}[\phi],k,\theta\circ \phi^{(-1)}]\circ \phi
      \]
      from $\mathcal{A}^{1,\alpha}_{\partial \Omega}\times \mathbb{C}\times C^{1,\alpha}(\partial\Omega)$ to  $C^{1,\alpha}(\partial\Omega)$ are real analytic. Moreover, the map
     \[
     k \mapsto (1-i \mathrm{Re} k)\,,
     \]
     from $\mathbb{C}$ to itself is real analytic. We deduce that the map 
     \[
\mathcal{A}^{1,\alpha}_{\partial \Omega}\times \mathbb{C}\times C^{1,\alpha}(\partial\Omega)\ni     (\phi,k,\theta)\mapsto \tilde\Lambda(\phi,k,\theta)\equiv \Lambda(\phi,k)[\theta]\in C^{1,\alpha}(\partial\Omega)
     \]
      is real analytic.  Since $\tilde\Lambda$ is linear and continuous with respect to the variable $\theta$, we have
\[
\Lambda(\phi,k) =d_\theta\tilde\Lambda(\phi,k,\theta)
\qquad\forall (\phi,k,\theta)\in \mathcal{A}^{1,\alpha}_{\partial \Omega}\times \mathbb{C}\times C^{1,\alpha}(\partial\Omega)\,.
\]
Since the right-hand side equals a partial Fr\'{e}chet differential of an analytic map, the right-hand side is analytic. Hence $\Lambda$ is analytic on $\mathcal{A}^{1,\alpha}_{\partial \Omega}\times \mathbb{C}\times C^{1,\alpha}(\partial\Omega)$ and, since it does not depend on $\theta$, we conclude that it is analytic on  $\mathcal{A}^{1,\alpha}_{\partial \Omega}\times \mathbb{C}$. 
     \end{proof}

  Now we find convenient to introduce the  set  
\[
\mathbb{C}^+\equiv \{k\in \mathbb{C} \colon \mathrm{Im} k\geq 0\}
\]     
of complex numbers with nonnegative imaginary part. 
In the following proposition we see that $\Lambda(\phi,k)$ is an isomorphism for all $(\phi,k)\in \mathcal{A}^{1,\alpha}_{\partial \Omega} \times \mathbb{C}^+$.
  \begin{proposition}\label{prop:iso}
  Let $\alpha$, $\Omega$ be as in \eqref{Omega_def}. For all $(\phi,k)\in \mathcal{A}^{1,\alpha}_{\partial \Omega} \times \mathbb{C}^+$ the operator    $\Lambda(\phi,k)$ is an isomorphism (i.e. a linear homeomorphism) from $C^{1,\alpha}(\partial \Omega)$ to itself.
  \end{proposition}
 \begin{proof} Since $\Lambda(\phi,k)$ is linear and continuous it suffices to show that it is bijective and then, by the open mapping theorem, we deduce that it is an isomorphism. The fact that $\Lambda(\phi,k)$ is a bijection follows by Theorem \ref{thm:varexist} and by noting that the map from $C^{1,\alpha}(\phi(\partial \Omega))$ to  $C^{1,\alpha}(\partial \Omega)$ that takes $\mu$ to $\theta\equiv\mu\circ\phi$ is a bijection.
 \end{proof}
 
 By Proposition \ref{prop:iso} it makes sense to define the map
 \[
 \mathcal{A}^{1,\alpha}_{\partial \Omega} \times \mathbb{C}^+\times C^{1,\alpha}(\partial \Omega)\ni(\phi,k,g)\mapsto\theta[\phi,k,g]\in C^{1,\alpha}(\partial \Omega)
 \]
  that takes a triple $(\phi,k,g)$ to the unique solution $\theta[\phi,k,g]$ of equation \eqref{Linteq}.
 We now prove that the map above is real analytic. Since $\mathbb{C}^+$ is not open, we clarify that this means that the map  has a real analytic continuation on an open neighborhood of every  $(\phi_0,k_0,g_0)\in \mathcal{A}^{1,\alpha}_{\partial \Omega} \times \mathbb{C}^+\times C^{1,\alpha}(\partial \Omega)$.
 
  \begin{proposition}\label{prop:antheta}
   Let $\alpha$, $\Omega$ be as in \eqref{Omega_def}. Then the map from $\mathcal{A}^{1,\alpha}_{\partial \Omega} \times \mathbb{C}^+ \times C^{1,\alpha}(\partial \Omega)$ to $C^{1,\alpha}(\partial \Omega)$ that takes $(\phi,k,g)$ to $\theta[\phi,k,g]$ is real analytic.
  \end{proposition}
       \begin{proof}
       Since the map that takes an invertible operator to its inverse is real analytic, Propositions \ref{prop:Lmbdan} and \ref{prop:iso} imply that $(\phi,k)\mapsto \Lambda(\phi,k)^{(-1)}$ is real analytic from  the product $\mathcal{A}^{1,\alpha}_{\partial \Omega} \times \mathbb{C}^+$ to $\mathcal{L}\left(C^{1,\alpha}(\partial \Omega),C^{1,\alpha}(\partial \Omega)\right)$. Then, since the evaluation map from $\mathcal{L}\left(C^{1,\alpha}(\partial \Omega),C^{1,\alpha}(\partial \Omega)\right)\times C^{1,\alpha}(\partial \Omega)$ to $C^{1,\alpha}(\partial \Omega)$ is bilinear and continuous we conclude that 
       \[
       (\phi,k,g)\mapsto \theta[\phi,k,g]=\Lambda(\phi,k)^{(-1)}[g]
       \] 
       is real analytic from  $\mathcal{A}^{1,\alpha}_{\partial \Omega} \times \mathbb{C}^+ \times C^{1,\alpha}(\partial \Omega)$ to $C^{1,\alpha}(\partial \Omega)$.
            \end{proof}

   \section{Analysis of the solution of problem (\ref{eq:obstaclepb}) and of associated functionals} \label{s:functionals}
We are now ready to exploit the intermediate result of Proposition \ref{prop:antheta} on the solutions of the equivalent integral equation \eqref{Linteq} to prove our main theorems. In particular,  Proposition \ref{prop:existchange} gives a representation of the solution of problem \eqref{eq:obstaclepb} by means 
of layer potentials with a density that, by Proposition \ref{prop:antheta},  depends analytically upon $(\phi,k,g)$. Then we can use Proposition \ref{prop:antheta} to prove a series of results on the 
analyticity of functions related to problem \eqref{eq:obstaclepb}. We start with a result on the analyticity of the solution $u[\phi,k,g]$.
\begin{theorem} \label{thm:an1}
 Let $\alpha$, $\Omega$ be as in \eqref{Omega_def}. Let $\Omega'$ be a bounded open subset of $\mathbb{R}^{3}$. Let $\mathcal{A}^{1,\alpha}_{\partial \Omega,\Omega'}$ be the open subset of $\mathcal{A}^{1,\alpha}_{\partial \Omega}$ consisting of  the functions $\phi$ such that  
  \[
 \overline{\Omega'}\subseteq \mathbb{E}[{\phi}] \,.
 \]
Then there exists a real analytic map $U_{\Omega'}$ from $\mathcal{A}^{1,\alpha}_{\partial \Omega,\Omega'}\times \mathbb{C}^+\times C^{1,\alpha}(\partial \Omega)$ to $C^2(\overline{\Omega'})$  such that
 \[
 u[\phi,k,g]_{|\overline{\Omega'}}=U_{\Omega'}[\phi,k,g] \qquad \qquad \forall  (\phi,k,g) \in\mathcal{A}^{1,\alpha}_{\partial \Omega,\Omega'}  \times \mathbb{C}^+ \times C^{1,\alpha}(\partial \Omega)\, .
 \]
\end{theorem}
\begin{proof} 
 By the definition of $\theta[\phi,k, g]$ and by Lemma \ref{lem:rajacon} (i), we have
 \[
 \begin{split}
  u[\phi,k,g](x)=&-\int_{\partial \Omega}
 \nu_{\mathbb{I}[\phi]}\circ \phi(s) \cdot 
DS(k,x-\phi(s)) \theta[\phi,k, g](s)\tilde{\sigma}[\phi](s)\,d\sigma_{s}\\&+(1-i \mathrm{Re} k) 
\int_{\partial \Omega}S(k,x-\phi(s))\theta[\phi,k, g](s)\tilde{\sigma}[\phi](s)\,d\sigma_{s} \qquad \forall x \in  \overline{\Omega'}\, ,
\end{split}
 \]
 for all $(\phi,k,g) \in \mathcal{A}^{1,\alpha}_{\partial \Omega,\Omega'}\times \mathbb{C}^+ \times C^{1,\alpha}(\partial \Omega)$. Then we define
  \[
 \begin{split}
U_{\Omega'}[\phi,k,g](x)\equiv&-\int_{\partial \Omega}
 \nu_{\mathbb{I}[\phi]}\circ \phi(s) \cdot 
DS(k,x-\phi(s)) \theta[\phi,k, g](s)\tilde{\sigma}[\phi](s)\,d\sigma_{s}\\&+(1-i \mathrm{Re} k) 
\int_{\partial \Omega}S(k,x-\phi(s))\theta[\phi,k, g](s)\tilde{\sigma}[\phi](s)\,d\sigma_{s} \qquad \forall x \in  \overline{\Omega'}\, ,
\end{split}
 \]
  for all $(\phi,k,g) \in \mathcal{A}^{1,\alpha}_{\partial \Omega,\Omega'}\times \mathbb{C}^+ \times C^{1,\alpha}(\partial \Omega)$. By  the properties of integral operators with real analytic kernels (cf.~Lanza de Cristoforis and Musolino \cite[Cor. 3.14]{LaMu13}), by Proposition \ref{prop:antheta}, and by Lemma \ref{lem:rajacon}, we verify that $U_{\Omega'}$ is a real analytic map from $\mathcal{A}^{1,\alpha}_{\partial \Omega,\Omega'}\times \mathbb{C}^+ \times C^{1,\alpha}(\partial \Omega)$ to $C^2(\overline{\Omega'})$.
 \end{proof}

  \begin{remark}
  We note that in Theorem \ref{thm:an1} we have chosen the target space $C^2(\overline{\Omega'})$  for the sake of simplicity. Indeed, by standard elliptic regularity theory, the solution $u[\phi,k,g]$ is real analytic in the interior of its domain. Therefore, we can easily replace the target space $C^2(\overline{\Omega'})$ with  $C^j(\overline{\Omega'})$ for any $j \in \mathbb{N}$ or even with a  suitable space of analytic functions.
  \end{remark}

Next we consider the normal derivative of the solution.
\begin{theorem}\label{thm:an2}
 Let $\alpha$, $\Omega$ be as in \eqref{Omega_def}. There exists a real analytic map $U^\#$ from $\mathcal{A}^{1,\alpha}_{\partial \Omega} \times \mathbb{C}^+ \times C^{1,\alpha}(\partial \Omega)$ to $C^{0,\alpha}(\partial \Omega)$ such that
  \[
\Big( \frac{\partial}{\partial \nu_{\mathbb{I}[\phi]}}u[\phi,k,g]\Big)\circ \phi =U^\#[\phi,k,g] \qquad \text{ on $\partial \Omega$} \qquad \forall  (\phi,k,g) \in \mathcal{A}^{1,\alpha}_{\partial \Omega} \times \mathbb{C}^+ \times C^{1,\alpha}(\partial \Omega)\, .
 \]
\end{theorem}
\begin{proof}
By Proposition \ref{prop:existchange} and by the jump formulas of the layer potentials, we have
\[
\begin{split}
\Big( \frac{\partial}{\partial \nu_{\mathbb{I}[\phi]}}u[\phi,k,g]\Big)\circ \phi=&\Big( \frac{\partial}{\partial \nu_{\mathbb{I}[\phi]}}w^-[\partial \mathbb{I}[\phi], k,\theta[\phi,k, g] \circ \phi^{(-1)}]\Big)\circ \phi\\&+(1-i \mathrm{Re} k) \Big(\frac{1}{2} \theta [\phi,k, g]+W^\ast[\partial \mathbb{I}[\phi],k,\theta [\phi,k, g]\circ \phi^{(-1)}]\circ \phi\Big)\, ,
\end{split}
\]
for all $(\phi,k,g) \in  \mathcal{A}^{1,\alpha}_{\partial \Omega} \times \mathbb{C}^+ \times C^{1,\alpha}(\partial \Omega)$. Thus, we define
\[
\begin{split}
U^\#[\phi,k,g]\equiv&\Big( \frac{\partial}{\partial \nu_{\mathbb{I}[\phi]}}w^-[\partial \mathbb{I}[\phi], k,\theta[\phi,k, g] \circ \phi^{(-1)}]\Big)\circ \phi\\&+(1-i \mathrm{Re} k) \Big(\frac{1}{2} \theta [\phi,k, g]+W^\ast[\partial \mathbb{I}[\phi],k,\theta [\phi,k, g]\circ \phi^{(-1)}]\circ \phi\Big)\, ,
\end{split}
\]
for all $(\phi,k,g) \in  \mathcal{A}^{1,\alpha}_{\partial \Omega} \times \mathbb{C}^+ \times C^{1,\alpha}(\partial \Omega)$. Then, by Lemma \ref{lem:anSLP} and Proposition \ref{prop:antheta}, we deduce the analyticity of $U^\#[\phi,k,g]$ and therefore the validity of the theorem.
\end{proof}

By Theorem \ref{thm:an2}, we deduce the validity of the following corollary on the regularity of the Dirichlet-to-Neumann operator (the proof can be effected by a standard argument of calculus in Banach spaces, see for example the last part of the proof of Proposition \ref{prop:Lmbdan}).

\begin{corollary}\label{cor:an} 
 Let $\alpha$, $\Omega$ be as in \eqref{Omega_def}.   There exists a real analytic map $\tilde{\mathcal{D}}$ from $\mathcal{A}^{1,\alpha}_{\partial \Omega} \times \mathbb{C}^+$   to $\mathcal{L}(C^{1,\alpha}(\partial \Omega),C^{0,\alpha}(\partial \Omega))$ such that
  \[
\mathcal{D}_{(\phi,k)} =\tilde{\mathcal{D}}[\phi,k]  \qquad \forall  (\phi,k) \in \mathcal{A}^{1,\alpha}_{\partial \Omega} \times \mathbb{C}^+\, .
 \]
\end{corollary}

Finally we consider the dependence of the far field pattern $u_\infty[\phi,k,g]$ (cf. \eqref{def:ff}) with respect to the perturbation of $(\phi,k,g)$.

\begin{theorem}\label{thm:an3}
 Let $\alpha$, $\Omega$ be as in \eqref{Omega_def}. There exists a real analytic map $U_\infty$ from $\mathcal{A}^{1,\alpha}_{\partial \Omega}\times \mathbb{C}^+ \times C^{1,\alpha}(\partial \Omega)$ to $C^2(\partial \mathbb{B}_3(0,1))$ such that
 \begin{equation}\label{thm:an3.eq1}
 u_\infty[\phi,k,g]=U_{\infty}[\phi,k,g] \quad \text{on $\partial \mathbb{B}_3(0,1)$} \qquad \forall  (\phi,k,g) \in  \mathcal{A}^{1,\alpha}_{\partial \Omega}\times \mathbb{C}^+ \times C^{1,\alpha}(\partial \Omega)\, .
 \end{equation}
\end{theorem}
\begin{proof}
Let $\rho>0$ and let $\mathcal{A}^{1,\alpha}_{\partial \Omega,\rho}$ be the open subset of $\mathcal{A}^{1,\alpha}_{\partial \Omega}$ of the functions  $\phi$ such that $\overline{\mathbb{I}[\phi]}\subseteq \mathbb{B}_3(0,\rho)$. By Theorem \ref{thm:an1}
with
\[
\Omega'\equiv \mathbb{B}_3(0,2\rho) \setminus \overline{\mathbb{B}_3(0,\rho)}\,
\]
(notice that $\mathcal{A}^{1,\alpha}_{\partial \Omega,\rho}$ is contained in the set $\mathcal{A}^{1,\alpha}_{\partial \Omega,\Omega'}$ of Theorem \ref{thm:an1}), by the continuity of the trace operator, and by standard calculus in Banach spaces, we deduce that  there exist two real analytic maps $V_1$ and  $V_2$ from $\mathcal{A}^{1,\alpha}_{\partial \Omega,\rho}\times \mathbb{C}^+ \times C^{1,\alpha}(\partial \Omega)$ to $C^0(\partial \mathbb{B}_3(0,\rho))$ such that
 \begin{equation}\label{thm:an3.eq2}
 u[\phi,k,g]_{|\partial \mathbb{B}_3(0,\rho)}=V_1 [\phi,k,g] \, , \qquad \frac{\partial u[\phi,k,g]}{\partial \nu_{\mathbb{B}_3(0,\rho)}}=V_2 [\phi,k,g]\, ,
 \end{equation}
 for all  $(\phi,k,g) \in \mathcal{A}^{1,\alpha}_{\partial \Omega,\rho} \times \mathbb{C}^+ \times C^{1,\alpha}(\partial \Omega)$. Then, having in mind the expression \eqref{def:ff} of the far field pattern, we set
\[
\begin{split}
U_\rho [\phi,k,g] (x)\equiv \frac{1}{4\pi}\int_{\partial \mathbb{B}_3(0,\rho)} \Bigg( V_1[\phi,k,g](y)\frac{\partial}{\partial \nu_{\mathbb{B}_3(0,\rho)}(y)}e^{-ikx\cdot y} -e^{-ikx\cdot y} V_2[\phi,k,g](y)\Bigg)& \, d\sigma_y \\ \forall x \in \partial \mathbb{B}_3(0,1)\, ,&
\end{split}
\]
for all $(\phi,k,g) \in \mathcal{A}^{1,\alpha}_{\partial \Omega,\rho} \times \mathbb{C}^+ \times C^{1,\alpha}(\partial \Omega)$.  By the properties of integral operators with real analytic kernels (cf.~Lanza de Cristoforis and Musolino \cite{LaMu13}), we deduce that  $U_\rho$ is a real analytic map  from $\mathcal{A}^{1,\alpha}_{\partial \Omega,\rho} \times \mathbb{C}^+ \times C^{1,\alpha}(\partial \Omega)$ to $C^2(\partial \mathbb{B}_3(0,1))$. Moreover, by equalities \eqref{thm:an3.eq2} and by \eqref{def:ff} (that does not depend on the specific choice of $R$), we have
\begin{equation}\label{thm:an3.eq3}
 u_\infty[\phi,k,g]=U_\rho[\phi,k,g] \quad \text{on $\partial \mathbb{B}_3(0,1)$} \qquad \forall  (\phi,k,g) \in  \mathcal{A}^{1,\alpha}_{\partial \Omega,\rho}\times \mathbb{C}^+ \times C^{1,\alpha}(\partial \Omega)\, .
 \end{equation}
Now, let $0<\rho_1\le \rho_2\le \dots$ be an increasing sequence of positive real numbers with $\lim_{j\to+\infty}\rho_j=+\infty$. Then $\mathcal{A}^{1,\alpha}_{\partial \Omega,\rho_1}\subseteq\mathcal{A}^{1,\alpha}_{\partial \Omega,\rho_2}\subseteq\dots$ is an increasing sequence of sets and we have $\cup_{j\ge 1}\mathcal{A}^{1,\alpha}_{\partial \Omega,\rho_j}=\mathcal{A}^{1,\alpha}_{\partial \Omega}$. By equality \eqref{thm:an3.eq3} we see that $U_{\rho_{j}}[\phi,k,g]=U_{\rho_{j+1}}[\phi,k,g]$ for all  $(\phi,k,g) \in  \mathcal{A}^{1,\alpha}_{\partial \Omega,\rho_j}\times \mathbb{C}^+ \times C^{1,\alpha}(\partial \Omega)$ and all $j\ge 1$. So, we are allowed to ``glue together'' the maps $U_{\rho_j}$ and define a map $U_\infty$ on the whole of $\mathcal{A}^{1,\alpha}_{\partial \Omega}\times \mathbb{C}^+ \times C^{1,\alpha}(\partial \Omega)$ by taking
\begin{equation}\label{thm:an3.eq4}
U_\infty[\phi,k,g]\equiv U_{\rho_j}[\phi,k,g]\quad\text{if  $\phi\in \mathcal{A}^{1,\alpha}_{\partial \Omega,\rho_j}$  for some $j\ge 1$}
 \end{equation}
and $(k,g) \in  \mathbb{C}^+ \times C^{1,\alpha}(\partial \Omega)$. By \eqref{thm:an3.eq3} and \eqref{thm:an3.eq4} we see that  \eqref{thm:an3.eq1} holds true and, in addition, $U_\infty$ inherits the real analyticity of the maps $U_{\rho_j}$.
\end{proof}

\subsection*{Acknowledgment}

The authors are members of the ``Gruppo Nazionale per l''Analisi Matematica, la Probabilit\`a e le loro Applicazioni'' (GNAMPA) of the ``Istituto Nazionale di Alta Matematica'' (INdAM). P.L.~and P.M.~acknowledge the support of the Project BIRD191739/19 ``Sensitivity analysis of partial differential equations in
the mathematical theory of electromagnetism'' of the University of Padova.  
P.M.~acknowledges the support of  the grant ``Challenges in Asymptotic and Shape Analysis - CASA''  of the Ca' Foscari University of Venice.  P.M.~also acknowledges the support from EU through the H2020-MSCA-RISE-2020 project EffectFact, 
Grant agreement ID: 101008140.


\vspace{5mm}
\begin{thebibliography}{10}
\expandafter\ifx\csname url\endcsname\relax
  \def\url#1{\texttt{#1}}\fi
\expandafter\ifx\csname urlprefix\endcsname\relax\def\urlprefix{URL }\fi
\expandafter\ifx\csname href\endcsname\relax
  \def\href#1#2{#2} \def\path#1{#1}\fi




\bibitem{Ch95} 
A.~Charalambopoulos, On the Fr\'echet differentiability of boundary integral operators in the inverse elastic scattering problem. Inverse Problems 11 (1995),   1137--1161.

 
 \bibitem{CoScZe18} 
A.~Cohen, C.~Schwab, and J.~Zech, Shape holomorphy of the stationary Navier-Stokes equations. SIAM J. Math. Anal. 50 (2018), no. 2, 1720--1752. 


\bibitem{CoKr13}
 D.~Colton and R.~Kress,  Integral equation methods in scattering theory. Reprint of the 1983 original. Classics in Applied Mathematics, 72. Society for Industrial and Applied Mathematics (SIAM), Philadelphia, PA, 2013.
 
 \bibitem{CoKr19} D.~Colton and R.~Kress, {Inverse acoustic and electromagnetic scattering theory.} Applied Mathematical Sciences 93, Springer, Cham,  2019.
 


\bibitem{CoLe12a} 
M.~Costabel and F.~Le Lou\"er, Shape derivatives of boundary integral operators in electromagnetic scattering. Part I: Shape differentiability of pseudo-homogeneous boundary integral operators. Integral Equations Oper.~Theory 72 (2012),  509--535.

 \bibitem{CoLe12b} 
 M.~Costabel and F.~Le Lou\"er, Shape derivatives of boundary integral operators in electromagnetic scattering. Part II: Application to scattering by a homogeneous dielectric obstacle. Integral Equations 
 Oper.~Theory 73 (2012),  17--48.
 
 \bibitem{Da08} M.~Dalla Riva,  Potential theoretic methods for the analysis of 
singularly perturbed problems in linearized elasticity,
PhD Thesis, University of  Padova, 2008.
 
 \bibitem{DaLa10a}
M.~Dalla~Riva and M.~Lanza~de Cristoforis,
 A perturbation result for the layer potentials of general second
  order differential operators with constant coefficients.
  J. Appl. Funct. Anal.  5 (2010), no. 1, 10--30.


\bibitem{DaLaMu21}
M.~Dalla Riva, M.~Lanza de Cristoforis, and P. Musolino,  Singularly Perturbed Boundary Value Problems: A Functional Analytic Approach. Springer Nature, Cham, 2021.

 
\bibitem{DaLuMuPu21}
M. Dalla Riva, P. Luzzini, P. Musolino, and R. Pukhtaievych, Dependence of effective properties upon regular perturbations.  In I. Andrianov, S. Gluzman, V. Mityushev, Editors, Mechanics and Physics of Structured Media. Asymptotic and Integral Equations Methods of Leonid Filshtinsky. Academic Press, Elsevier, London, 2022, pp.~271--301.
 
\bibitem{De85}
K.~Deimling, Nonlinear {F}unctional {A}nalysis, Springer-Verlag, Berlin, 1985.


\bibitem{Fo95}
G.~B. Folland.
 Introduction to partial differential equations.
 Princeton University Press, Princeton, NJ, second edition, 1995.



\bibitem{GiTr83}
D.~Gilbarg and N.S. Trudinger, Elliptic partial differential equations of second
  order, 2nd Edition, Vol. 224 of Grundlehren der Mathematischen Wissenschaften
  [Fundamental Principles of Mathematical Sciences], Springer-Verlag, Berlin,
  1983.
  
   \bibitem{HaKr04}
H.~Haddar and R.~Kress, On the Fr\'echet derivative for obstacle scattering with an impedance boundary condition. SIAM J. Appl. Math. 65 (2004), no. 1,  194--208. 

 \bibitem{HeSc21}
 F.~Henr\'iquez and C.~Schwab, Shape holomorphy of the Calder\'on projector for the Laplacian in $\mathbb{R}^2$. Integral Equations Operator Theory 93 (2021), no. 4, Paper No. 43, 40 pp. 
 
 
   \bibitem{HePi05}
A.~Henrot and M.~Pierre, Variation et optimisation de formes, Vol.~48 of
  Math\'ematiques \& Applications (Berlin) [Mathematics \& Applications],
  Springer, Berlin, 2005. 


\bibitem{He95} 
F.~Hettlich, Fr\'echet derivatives in inverse obstacle scattering. Inverse Problems 11 (1995), no. 2,  371--382. 

 \bibitem{JeScZe17} 
C.~Jerez-Hanckes, C.~Schwab, and J.~Zech, Electromagnetic wave scattering by random surfaces: shape holomorphy. Math. Models Methods Appl. Sci. 27 (2017), no. 12, 2229--2259.
  
 \bibitem{Ki93} 
 A.~Kirsch, The domain derivative and two applications in inverse scattering theory. Inverse Problems 9 (1993), no. 1,
   81--96.
   
 \bibitem{Ki21}
A.~Kirsch, {An introduction to the mathematical theory of inverse problems.} Applied Mathematical Sciences, 120. Springer, Cham, 2021.
  
  
  \bibitem{KrPa99}
R.~Kress and L.~P\"aiv\"arinta, On the far field in obstacle scattering. 
SIAM J. Appl. Math. 59  (1999), no. 4,  1413--1426. 
    
  \bibitem{La07}
M.~Lanza~de Cristoforis, Perturbation problems in potential theory, a
  functional analytic approach. J. Appl. Funct. Anal. 2  (2007), no. 3, 197--222.
  
  
\bibitem{LaMu13}
M. Lanza de Cristoforis and P. Musolino, A real analyticity result for a nonlinear integral operator, J. Int. Equ. Appl. 25 (2013), no. 1, 21--46.

   
   \bibitem{LaRo04}
M.~Lanza {de}~Cristoforis and L.~Rossi, Real analytic dependence of simple and
  double layer potentials upon perturbation of the support and of the density.
  J. Integral Equations Appl. 16 (2004), no. 2, 137--174.
  

  
\bibitem{LaRo08}
M.~Lanza~de Cristoforis and L.~Rossi, Real analytic dependence of simple and
  double layer potentials for the {H}elmholtz equation upon perturbation of the
  support and of the density, in: Analytic methods of analysis and differential
  equations: {AMADE} 2006, Camb. Sci. Publ., Cambridge, 2008, pp. 193--220.
  
 
 
 \bibitem{La12}
M.~Lanza~de Cristoforis, Simple Neumann eigenvalues for the Laplace operator in a domain with a small hole. A functional analytic approach. Rev. Mat. Complut. 25 (2012), no. 2, 369--412.
 
 
\bibitem{LuMu20}
P. Luzzini and P. Musolino, Perturbation analysis of the effective conductivity of a periodic composite. Netw. Heterog. Media, 15 (2020), no. 4, 581--603.


\bibitem{LuMuPu19}
P. Luzzini, P. Musolino, and R. Pukhtaievych, Shape analysis of the longitudinal flow along a periodic array of cylinders. J. Math. Anal. Appl., 477 (2019), no. 2, 1369--1395.


\bibitem{Le12}
F.~Le Lou\"er, On the Fr\'echet derivative in elastic obstacle scattering. SIAM J. Appl. Math. 72 (2012),  no. 5,  1493--1507.

\bibitem{NoSo13}
A.~A. Novotny and J.~Soko{\l}{o}wski, Topological derivatives in shape
  optimization, Interaction of Mechanics and Mathematics, Springer, Heidelberg,
  2013.
    
    

\bibitem{Po94}
R.~Potthast, 
Fr\'echet differentiability of boundary integral operators in inverse acoustic scattering. Inverse Problems 10 (1994), no. 2, 431--447. 




\bibitem{Po96a} 
R.~Potthast, Fr\'echet differentiability of the solution to the acoustic Neumann scattering problem with respect to the domain. J. Inverse Ill-Posed Probl.  4 (1996),  no. 1,  67--84.
 
\bibitem{Po96b}  
R.~Potthast, Domain derivatives in electromagnetic scattering. Math. Methods Appl. Sci.  19 (1996),  no. 15, 1157--1175. 


\bibitem{SoZo92}
J.~Soko{\l}{o}wski and J.-P. Zol\'esio, Introduction to shape optimization.
  {S}hape sensitivity analysis, Vol.~16 of Springer Series in Computational
  Mathematics, Springer-Verlag, Berlin, 1992.



\end{thebibliography}
\end{document}